\documentclass[11pt]{article}

\usepackage{epsfig,amstext,amsmath,amsfonts,amssymb,float}
\usepackage{graphicx}
\usepackage{fancybox,array}

\newcommand{\disp}{\displaystyle}
\newcommand{\real}{\mathbb{R}}

\newcommand{\complex}{\mathbb{C}}

\newcommand{\ffi}{\varphi}
\newcommand{\rn}{\real^N}

\newcommand{\intrn}{\int_{\rn}}

\newcommand{\wto}{\rightharpoonup}

\newcommand{\wt}{\widetilde}
\newcommand{\diff}{\,\mathrm{d}}

\newcommand{\s}{\sigma}
\newcommand{\e}{\mathrm{e}}
\newcommand{\x}{\times}
\newcommand{\sm}{\setminus}

\newtheorem{theorem}{Theorem}
\newtheorem{lemma}{Lemma}
\newtheorem{proposition}{Proposition}
\newtheorem{corollary}{Corollary}
\newtheorem{remark}{Remark}

\newenvironment{proof}
{\hspace{-3mm}{\it Proof.}}{\nolinebreak\hfill$\Box$}

\numberwithin{equation}{section}

\begin{document}

\title{An inhomogeneous, $L^2$ critical, nonlinear Schr\"odinger equation}

\author{Fran\c cois Genoud \\
\small{\it Department of Mathematics and}\\ 
\small{\it The Maxwell Institute for Mathematical Sciences}\\
\small{\it Heriot-Watt University}\\
\small{\it Edinburgh, EH14 4AS, Scotland}\\
\small{F.Genoud@hw.ac.uk}}

\maketitle

\begin{abstract}
An inhomogeneous nonlinear Schr\"odinger equation is considered, that is invariant under $L^2$ scaling. The sharp condition for global existence of $H^1$ solutions is established, involving the $L^2$ norm of the ground state of the stationary equation. Strong instability of standing waves is proved by constructing self-similar solutions blowing up in finite time. 
\end{abstract}

{\em This paper will appear in the 	
Zeitschrift f\"ur Analysis und ihre Anwendungen.}

\section{Introduction}

The purpose of this note is to point out the case of an inhomogeneous nonlinear 
Schr\"odinger equation having $L^2$ scaling invariance. Namely, we consider the Cauchy problem
\begin{equation}\label{nls}\tag{NLS}
i\partial_t\phi + \Delta\phi + |x|^{-b}|\phi|^{2\s}\phi = 0, \quad 
\phi(0,\cdot)=\phi_0\in H^1(\rn)
\end{equation}
with $\s=(2-b)/N$, in any dimension $N\ge1$. Here and henceforth, $H^1(\rn)$ denotes the Sobolev space of complex-valued functions $H^1(\rn,\complex)$, with its usual norm. We suppose that $0<b<\min\{2,N\}$. The case $b=0$ is the classical (focusing) nonlinear Schr\"odinger equation with $L^2$ critical nonlinearity. In the above setting, it turns out that \eqref{nls} is also invariant under the $L^2$ scaling
\begin{equation}\label{scaling}
\phi\to\phi_\lambda(t,x):=\lambda^{N/2}\phi(\lambda^2 t,\lambda x), \quad
\phi_0\to(\phi_0)_\lambda(x):=\lambda^{N/2}\phi_0(\lambda x) \quad
\text{for} \ \lambda>0.
\end{equation}

We came across this (modified) critical nonlinearity for \eqref{nls} while studying stability
of standing waves for some classes of nonlinear Schr\"odinger equations, where 
\eqref{nls} arises both as a model and a limiting case, see \cite{gs,these,g}. In particular, the Cauchy problem \eqref{nls} is studied there, and it is found that, for
$0<b<\min\{2,N\}$, it is well-posed in $H^1(\rn)$,
\begin{align*}
&\text{locally if} \ 0<\s<\wt{2}
:=\left\{ \begin{array}{ll}
(2-b)/(N-2) & \text{if} \ N\ge3,\\
\infty & \text{if} \ N\in\{1,2\};\\
\end{array}\right.\\
&\text{globally for small initial conditions if}  \  0<\s<\wt{2};\\
&\text{globally for any initial condition in} \ H^1(\rn) \ \text{if} \ 0<\s<\frac{2-b}{N}.
\end{align*}
Theorem~\ref{global} below answers the natural question: in the limit case $\s=(2-b)/N$, how small should the initial condition be to have global existence?
We consider here strong solutions $\phi=\phi(t,x)\in C^0_t H^1_x ([0,T)\x\rn)$ for some $T>0$, and the notion of well-posedness as defined in \cite{cazenave}. Our notation for the space-time function spaces comes from \cite{tao}. We may simply denote by $\phi(t)\in H^1(\rn)$ the function $x\to\phi(t,x)$. The solution is called global (in time) if we can take $T=\infty$. If it is not the case, the blowup alternative states that $\Vert\phi(t)\Vert_{H^1}\to\infty$ as $t\uparrow T$. Moreover, we have conservation of the $L^2$ norm along the flow of \eqref{nls}, 
$$
\Vert \phi(t)\Vert_{L^2_x}=\Vert \phi_0\Vert_{L^2_x} \quad \text{for all} \ t\in[0,T),
$$
and of the energy
\begin{equation}\label{energy}
E(\phi(t)):=\intrn|\nabla\phi(t)|^2\diff x - \frac{1}{\s+1}\intrn|x|^{-b}|\phi(t)|^{2\s+2}\diff x
=E(\phi_0)  \quad \text{for all} \ t\in[0,T).
\end{equation}
Also, the $L^2$ norm of $\phi(t)$ is invariant under the transformation \eqref{scaling}, 
$\Vert\phi(t)\Vert_{L^2_x}=\Vert\phi_\lambda(t)\Vert_{L^2_x}$. This is why it is called the 
$L^2$ scaling.

\medskip

A standing wave for \eqref{nls} is a (global) solution of the form 
$\ffi_\omega(t,x)=\e^{i\omega^2 t}u_\omega(x)$ for some $\omega\in\real$, with 
$u_\omega\in H^1(\rn)$ satisfying the stationary equation
\begin{equation}\label{snls}\tag{$\mathrm{E}_\omega$}
\Delta u - \omega^2 u + |x|^{-b}|u|^{2\s}u=0.
\end{equation}

In \cite{gs,these,g}, we were concerned with bifurcation and orbital
stability of standing waves for nonlinear Schr\"odinger equations with inhomogeneous nonlinearities of the form $V(x)|\phi|^{2\s}\phi$ with $V(x)\sim |x|^{-b}$ at infinity or around the origin. These equations have important applications in nonlinear optics (see \cite{g}). The limiting problem \eqref{nls} turned out to play a central role in our analysis. For this model case, a global branch of positive solutions of \eqref{snls} is simply given by the mapping $u\in C^1((0,\infty),H^1(\rn))$,
\begin{equation}\label{branch}
\omega \mapsto u_\omega(x)= u(\omega)(x):=\omega^{\frac{2-b}{2\s}}u_1(\omega x),
\end{equation}
where $u_1$ is the unique positive radial solution (ground state) of \eqref{snls} with 
$\omega=1$. The existence of the ground state is proved in \cite{gs,these} by variational methods in dimension $N\ge2$, and in \cite{g} for $N=1$. Uniqueness is a delicate problem, handled in dimension $N\ge3$ by a theorem of Yanagida \cite{yan}
(see \cite{these}), in dimension $N=2$ by a shooting argument \cite{uniqueness}, and in dimension $N=1$ by the method of horizontal separation of graphs of Peletier and Serrin \cite{pelser}, as used in \cite{toland}. These existence and uniqueness results hold for $0<b<\min\{2,N\}$ and $0<\s<\wt{2}$.

\medskip

Using the general theory of orbital stability of Grillakis, Shatah and Strauss \cite{gss}, we obtained in \cite{gs,these,g} various stability/instability results for general nonlinearities $V(x)|\phi|^{2\s}\phi$ by studying the monotonicity of the $L^2$ norm of the standing waves, as a function of $\omega>0$. It turned out that $\s=(2-b)/N$ is a threshold for stability in the regimes we considered. For this value of $\s$, we could not determine if the standing waves are stable or not, even in the model case $V(x)=|x|^{-b}$. In fact, if 
$\s=(2-b)/N$, we have  
$\Vert u_\omega\Vert_{L^2}=\Vert u_1\Vert_{L^2}$ along the curve of solutions 
\eqref{branch}, for $u_\omega$ is then an $L^2$ scaling of $u_1$. In 
Section~\ref{instab}, we prove a strong instability result for standing waves of \eqref{nls}, without requiring that $u_\omega$ be the ground state of \eqref{snls}.

\medskip

Section~\ref{criticalmass} is devoted to a sharp global existence result in the spirit of Weinstein \cite{weinstein}. For $\s=(2-b)/N$ we prove that the solutions of \eqref{nls} are global in time provided $\Vert\phi_0\Vert_{L^2}<\Vert\psi\Vert_{L^2}$, where $\psi$ is the ground state of ($\mathrm{E}_1$). This is done by computing the best constant for an interpolation inequality. The sharpness of the result is proved in Section~\ref{instab} where we construct self-similar solutions blowing up in finite time, in particular with the critical mass $\Vert\psi\Vert_{L^2}$.

\medskip  

Related results for inhomogeneous nonlinear Schr\"odinger equations can be found in the literature, see for instance \cite{merle96} and \cite{chenguo2007}. However, no one seems to have noticed the possibility of $L^2$ scaling invariance. The results established here use basic ideas going back to \cite{weinstein,weinstein2}. The classical $L^2$ critical case ($b=0$) has been studied extensively, and in particular the properties of the blowup solutions are quite well-known (see \cite{raphael} for a survey). The case $b\neq0$ certainly deserves further investigation.

\medskip

\noindent{\bf Notation.} In Section~\ref{criticalmass} we work in the Sobolev space of real-valued functions $H:=H^1(\rn,\real)$. We use the shorthand notation 
$\Vert\cdot\Vert_p:=\Vert\cdot\Vert_{L^p}$ for the usual Lebesgue norms throughout.

\section{Critical mass and global existence}\label{criticalmass}

We start by solving the minimization problem
\begin{equation}\label{mp}
\inf_{u\in H\sm\{0\}}J(u)
\end{equation}
where $J:H\sm\{0\}\to\real$ is the Weinstein functional defined by
\begin{equation}\label{J}
J(u)=J_{N,b}(u)=\frac{\Vert\nabla u\Vert_2^2\Vert u\Vert_2^{2\s}}{I(u)} \quad
\text{for} \ \s=\frac{2-b}{N},
\end{equation}
with
\begin{equation}\label{I}
I(u)=\intrn|x|^{-b}|u|^{2\s+2}\diff x.
\end{equation}

\begin{lemma}\label{wsc}
For $N\geq1$, $0<b<\min\{2,N\}$ and $0<\s<\wt{2}$, the functional $I:H\to\real$ defined in \eqref{I} is of class $C^1(H,\real)$ and is weakly sequentially continuous. In particular, it follows that $J\in C^1(H\sm\{0\},\real)$.
\end{lemma}

\begin{proof}
See \cite[Section~2.1]{gs} and \cite[Section~1.1]{these} for $N\ge2$, \cite[Section~2]{g} for $N=1$.
\end{proof}

\begin{proposition}\label{min}
Let $N\geq1$, $0<b<\min\{2,N\}$ and $\sigma=(2-b)/N$. There exists a positive radial function $\psi\in H$ such that:
\item[(i)] $\psi$ is a minimizer for \eqref{mp}, that is,
$J_{N,b}(\psi)=\disp\inf_{u\in H\sm\{0\}}J_{N,b}(u)$;
\item[(ii)] $\psi$ is the unique ground state of 
$\mathrm{(E}_{\sqrt{\s}}\mathrm{)}$.

\noindent
Furthermore, the minimum value is 
$J_{N,b}(\psi)=\disp\frac{\Vert\psi\Vert_2^{2\s}}{\s+1}
=\frac{\Vert\psi\Vert_2^{\frac{4-2b}{N}}}{\frac{2-b}{N}+1}$.
\end{proposition}

\begin{proof} We follow Weinstein \cite{weinstein}. 
Let $\{u_n\}\subset H\sm\{0\}$ be a minimizing sequence for \eqref{mp}: 
$$
J(u_n)\to m:=\inf J\ge0 \quad \text{as} \ n\to\infty.
$$ 
Clearly, we can choose $u_n\ge0$. Moreover, by Schwarz symmetrization 
(see \cite[p.146]{gs}) we can suppose that $u_n$ is radial and radially non-increasing for all $n$. It follows from the structure of $J=J_{N,b}$ that $J$ is invariant under the scaling $u\to u_{\lambda,\mu}(x):=\lambda u(\mu x)$, $\lambda,\mu>0$. (This is not the case for 
$\sigma\neq(2-b)/N$.) This allows us to choose $u_n$ such that 
$$
\Vert\nabla u_n\Vert_2=\Vert u_n\Vert_2=1 \quad \text{for all} \ n.
$$
Hence there exists $u^*\in H$ such that, up to a subsequence, $u_n\wto u^*$ weakly in $H$. Furthermore, $u^*$ is non-negative, spherically symmetric, radially non-increasing, and
\begin{equation}\label{norms}
\Vert\nabla u^*\Vert_2\le 1 \quad \text{and} \quad \Vert u^*\Vert_2\le 1.
\end{equation}
Now by Lemma~\ref{wsc} and \eqref{norms} we have
\begin{equation}\label{limit}
m=\lim J(u_n)=\lim\frac{1}{I(u_n)}=\frac{1}{I(u^*)}\ge J(u^*)
\end{equation}
so that, in fact, $J(u^*)=m$ and $\Vert\nabla u^*\Vert_2=\Vert u^*\Vert_2=1$. In particular, $u_n\to u^*$ strongly in $H$. (Note that \eqref{limit} prevents $u^*=0$.) This concludes the proof of (i).

\medskip

To show that $\psi$ can be chosen so as to satisfy ($\mathrm{E}_{\sqrt{\s}}$), we first remark that $u^*$ is a solution of the Euler-Lagrange equation corresponding to 
\eqref{mp}, which reads
$$
\Delta u^* -\s u^* + m(\s+1)|x|^{-b}(u^*)^{2\s+1}=0.
$$
Setting $u^*=[m(\s+1)]^{-1/2\s}\psi$, it follows that $\psi$ is a solution of 
($\mathrm{E}_{\sqrt{\s}}$). Furthermore, $\psi$ is positive and radial, so it is the unique ground state of ($\mathrm{E}_{\sqrt{\s}}$).
\end{proof}

\medskip

As an immediate consequence we have

\begin{corollary}\label{best}
$C_{N,b}:=\disp\frac{\frac{2-b}{N}+1}{\Vert\psi\Vert_2^{\frac{4-2b}{N}}}$ is the best constant for the inequality
\begin{equation}\label{interp}
\intrn|x|^{-b}|u|^{\frac{4-2b}{N}+2}\diff x\leq 
C\Vert\nabla u\Vert_2^2\Vert u\Vert_2^{\frac{4-2b}{N}}, \quad u\in H.
\end{equation}
\end{corollary}

\begin{remark} 
{\rm
Note that \eqref{interp} is a special case of the interpolation inequalities
obtained in \cite{CKN}.
}
\end{remark}

We now turn to the global existence result.

\begin{theorem}\label{global}
Set $\sigma=(2-b)/N$ and let $\psi$ be the ground state of 
$\mathrm{(E}_1\mathrm{)}$. If 
$$
\Vert\phi_0\Vert_2<\Vert\psi\Vert_2,
$$ 
the solution of \eqref{nls} is global and bounded in $H^1$.
\end{theorem}

\begin{proof} Local existence of solutions to \eqref{nls} is ensured by results in 
\cite{cazenave} (see \cite[Appendix~K]{gs} for precise statements and references). So the maximal solution $\phi(t,x)$ of \eqref{nls} with initial condition $\phi_0$ is defined on a time interval $[0,T)$ with $T\in(0,\infty]$. Moreover, we have the conservation laws
$$
E(\phi(t))=E(\phi_0) \quad \text{and} \quad \Vert\phi(t)\Vert_2=\Vert\phi_0\Vert_2
\quad \text{for all} \ t\in[0,T),
$$
where $E$ is defined in \eqref{energy}.
It is well-known since \cite{ginvel} that the boundedness of $\Vert\nabla\phi(t)\Vert_2$ is then sufficient to conclude global existence. Using the constants of motion, we have
\begin{align*}
\Vert\nabla\phi(t)\Vert_2^2 &=E(\phi(t))+\frac{1}{\s+1}\intrn|x|^{-b}|\phi(t)|^{2\s+2}\diff x\\
		&\le E(\phi_0)+\frac{C}{\s+1}\Vert\nabla\phi(t)\Vert_2^2\Vert \phi_0\Vert_2^{2\s},
\end{align*}
where $C=C_{N,b}>0$ is the constant given by Corollary~\ref{best}. Hence,
\begin{equation}\label{bound}
\left(1-\frac{C_{N,b}}{\frac{2-b}{N}+1}\Vert\phi_0\Vert_2^{\frac{4-2b}{N}}\right) 
\Vert\nabla\phi(t)\Vert_2^2\le E(\phi_0).
\end{equation}
Using the formula for $C_{N,b}$, it follows from \eqref{bound} that the solution is global if $\Vert\phi_0\Vert_2<\Vert\psi\Vert_2$ where $\psi$ is the ground state of 
($\mathrm{E}_{\sqrt{\s}}$). But for $\sigma=(2-b)/N$, ($\mathrm{E}_{\sqrt{\s}}$) is transformed into 
($\mathrm{E}_1$) by the scaling
$$
\psi\to \psi_{\lambda^{-1}}(x)=\lambda^{-N/2}\psi(\lambda^{-1}x) \quad 
\text{with} \ \lambda=\sqrt{\s}.
$$
Since this transformation leaves the $L^2$ norm unchanged, we can indeed choose 
$\psi$ to be the ground state of ($\mathrm{E}_1$). The proof is complete.
\end{proof}

\begin{remark} 
{\rm
We call $\Vert\psi\Vert_2$ the {\em critical mass} for \eqref{nls}. As we show below, the condition for global existence given by 
Theorem~\ref{global} is sharp in the sense that we can find solutions with critical mass which blow up in finite time.
}
\end{remark}

\section{Instability of standing waves}\label{instab}

It is a lenghty but straightforward calculation to show that \eqref{nls} is invariant under the pseudoconformal transformation, as defined in \cite[Section~6.7]{cazenave}.
Namely, for any $a\in\real$, if $\phi(s,y)\in C^0_s H^1_y ([0,S)\x\rn)$ is a solution to \eqref{nls} (with the obvious modification of the variables), then the function
$\phi_a(t,x)\in C^0_t H^1_x ([0,T)\x\rn)$ defined by
\begin{equation}\label{pseudo}
\phi_a(t,x)=(1-at)^{-\frac{N}{2}}\e^{-i\frac{a|x|^2}{4(1-at)}}
\phi\left(\frac{t}{1-at},\frac{x}{1-at}\right) \quad
\text{with} \quad
T=\left\{ \begin{array}{lr}
\infty & \text{if} \ aS\le -1\\
\frac{S}{1+aS} & \text{if} \ aS> -1\\
\end{array}\right.
\end{equation}
is also a solution. The fact that \eqref{nls} with $\s=(2-b)/N$ behaves nicely under 
\eqref{pseudo} when $b>0$ is closely related to the $L^2$ scaling invariance of the equation. In fact, the pseudoconformal transformation conserves the $L^2$ norm:
$$
\Vert \phi_a(t)\Vert_2\equiv \Vert \phi(s)\Vert_2.
$$

Using \eqref{pseudo}, we now show that all standing waves for \eqref{nls} with 
$\s=(2-b)/N$ are strongly unstable in the following sense. By scaling, it is enough to consider the case $\omega=1$.

\begin{theorem}\label{blow}
Let $u\in H$ be a nontrivial solution of $\mathrm{(E}_1\mathrm{)}$. 
For any $\delta>0$ there exists a solution
$\ffi\in C^0_t H^1_x ([0,T)\x\rn)$ of \eqref{nls} such that 
$\Vert\ffi(0)-u\Vert_{H^1}<\delta$ and $\Vert\ffi(t)\Vert_{H^1}\to\infty$ as 
$t\uparrow T$.
\end{theorem}

\begin{proof} Let $a>0$ to be tuned later. We apply the transformation \eqref{pseudo} to the standing wave $\phi(t,x)=\e^{i t}u(x)$, defining $\ffi\in C^0_tH^1_x ([0,a^{-1})\x\rn)$
by ($S=\infty$ for $\phi$):
\begin{equation}\label{selfsim}
\ffi(t,x)=(1-at)^{-\frac{N}{2}}\e^{-i\frac{a|x|^2}{4(1-at)}}
\e^{i\frac{t}{1-at}} u\left(\frac{x}{1-at}\right).
\end{equation}
It is easy to check that 
$$
(1-at)\Vert\nabla\ffi(t)\Vert_2\to\Vert\nabla u\Vert_2 \quad \text{as} \ t\uparrow a^{-1}
$$ 
and so $\ffi$ blows up at finite time $T:=a^{-1}$. Furthermore, 
$\ffi(0,x)=\e^{-i\frac{a|x|^2}{4}}u(x)$ and we have:
\begin{align}
\Vert\ffi(0)-u\Vert_2^2 &=\intrn|\e^{-i\frac{a|x|^2}{4}}-1|^2u(x)^2\diff x\label{est1} \\
\text{and} \quad
\Vert\nabla\ffi(0)-\nabla u\Vert_2^2 &=
\intrn |\e^{-i\frac{a|x|^2}{4}}-1|^2|\nabla u(x)|^2+\frac{a^2}{4}|x|^2u(x)^2\diff x.\label{est2}
\end{align}
It is standard to show that $u$ decays exponentially and it follows by dominated convergence that both \eqref{est1} and \eqref{est2} go to zero as $a\to0$. Hence, for any $\delta>0$, there is $a_\delta>0$ such that $\Vert\ffi(0)-u\Vert_{H^1}<\delta$ whenever $0<a<a_\delta$. This concludes the proof.
\end{proof}

\begin{remark}
{\rm
\item[(i)] We know precisely the blowup rate of $\ffi$,
$$
\Vert\ffi(t)\Vert_{H^1}\sim (1-at)^{-1} \quad \text{and} \quad
\Vert\ffi(t)\Vert_\infty\sim(1-at)^{-N/2} \quad  \text{as} \ t\uparrow a^{-1}.
$$ 
\item[(ii)] The type of solutions constructed in \eqref{selfsim} are often called `self-similar' in the literature. In fact, the modulus $|\ffi(t,x)|=(1-at)^{-N/2}|u(x/(1-at))|$ presents a self-similar profile in the usual sense: at any time $t$, there is a scaling parameter 
$\lambda(t)>0$ such that $|u(x)|=\lambda(t)^{N/2}|\ffi(t,\lambda(t)x)|$. Thus $|\ffi(t)|$ retains the shape of $|u|$ while blowing up.
}
\end{remark}

\begin{corollary}
There exists a solution of \eqref{nls} with critical mass that blows up in finite time.
\end{corollary}

\begin{proof} Take $\ffi$ defined by \eqref{selfsim} with $u=\psi$, the ground state of
$\mathrm{(E}_1\mathrm{)}$.
\end{proof}

\begin{remark} 
{\rm
Note that \eqref{selfsim} yields blowup solutions with self-similar profiles corresponding to any solution of $\mathrm{(E}_1\mathrm{)}$. In particular, it follows by 
Theorem~\ref{global} that $\psi$ is the solution of $\mathrm{(E}_1\mathrm{)}$ with minimal $L^2$ norm, as is well-known in the case $b=0$. 
}
\end{remark}

\end{document}